\documentclass[11pt]{amsart}
\usepackage{mathrsfs}
\usepackage{cases}
\usepackage{amssymb}
\usepackage{amsfonts}
\usepackage{amssymb,amsfonts}

\usepackage{amsmath,amssymb,amsthm,hyperref}
\allowdisplaybreaks

\begin{document}
\theoremstyle{plain}
\newtheorem{thm}{Theorem}[section]
\newtheorem{theorem}[thm]{Theorem}
\newtheorem{lemma}[thm]{Lemma}
\newtheorem{corollary}[thm]{Corollary}
\newtheorem{corollary*}[thm]{Corollary*}
\newtheorem{proposition}[thm]{Proposition}
\newtheorem{proposition*}[thm]{Proposition*}
\newtheorem{conjecture}[thm]{Conjecture}
%%%%%%%%%%%%%%%%%%%% Text roman %%%%%%%%%%%%%%%%%%%%%%%%%%%%%
\theoremstyle{definition}
\newtheorem{construction}{Construction}
\newtheorem{notations}[thm]{Notations}
\newtheorem{question}[thm]{Question}
\newtheorem{problem}[thm]{Problem}
\newtheorem{remark}[thm]{Remark}
\newtheorem{remarks}[thm]{Remarks}
\newtheorem{definition}[thm]{Definition}
\newtheorem{claim}[thm]{Claim}
\newtheorem{assumption}[thm]{Assumption}
\newtheorem{assumptions}[thm]{Assumptions}
\newtheorem{properties}[thm]{Properties}
\newtheorem{example}[thm]{Example}
\newtheorem{comments}[thm]{Comments}
\newtheorem{blank}[thm]{}
\newtheorem{observation}[thm]{Observation}
\newtheorem{defn-thm}[thm]{Definition-Theorem}

\newcommand{\sM}{{\mathcal M}}

%%%%%%%%%%%%%%%%%%%%%%%%%%%%%%%%%%%%%%%%%%%%%%%%%%%%%%%%%%%%%%

\title{Topological strings, quiver varieties and Rogers-Ramanujan identities}
        \author{Shengmao Zhu}
       \address{Center of Mathematical Sciences, Zhejiang University, Hangzhou, Zhejiang 310027, China}
        \email{szhu@zju.edu.cn}
\maketitle

\begin{abstract} Motivated by some recent works on BPS invariants of
open strings/knot invariants, we guess there may be a general
correspondence between the Ooguri-Vafa invariants of toric
Calabi-Yau 3-folds and cohomologies of Nakajima quiver varieties. In
this short note, we provide a toy model to explain this
correspondence. More precisely, we study the topological open string
model of $\mathbb{C}^3$ with one Aganagic-Vafa brane
$\mathcal{D}_\tau$, and we show that, when $\tau\leq 0$, its
Ooguri-Vafa invariants are given by the Betti numbers of certain
quiver variety. Moreover, the existence of Ooguri-Vafa invariants
implies an infinite product formula. In particular, we find that the
$\tau=1$ case of such infinite product formula is closely related to
the celebrated Rogers-Ramanujan identities.
\\ \noindent{\bf MSC classes}: 14N35; 14N10; 11P84; 05E05.
\end{abstract}

\section{Introduction}
Topological string theory is the topological sector of  superstring
theory \cite{Witten}. In mathematics, we use Gromov-Witten theory to
describe the topological string theory, see \cite{HKKPTVVZ} for a
review. Topological string amplitude is the generating function of
Gromov-Witten invariants which are usually rational numbers
according to their definitions \cite{BF,LT}. In 1998, Gopakumar and
Vafa \cite{GV1} found that topological string amplitude is also the
generating function of a series of integer-valued invariants related
to BPS counting in M-theory. Later, Ooguri and Vafa \cite{OV}
extended the above result to open string case, we name the
corresponding integer-valued invariants as OV invariants.
Furthermore, the OV invariants are further refined by Labasitida,
Mari\~no and Vafa in \cite{LMV}, the resulted invariants are called
LMOV invariants \cite{LP}, which have been studied by many
literatures, see \cite{LZ,MMMS} for the recent approaches.

A central question in topological string theory is how to define the
GV/OV/LMOV invariants directly.  There have been many works, for
examples \cite{HST,KiL,PT,IP,MT}, devoted to the definition of GV
invariants. However, to the author's knowledge, no direct related
works study the definition of OV/LMOV invariants. But there are some
attempts to explain the integrality of OV invariants through
different mathematical models. In \cite{KS}, Kucharski and Sulkowski
related the OV invariants to the combinatorics on words. In the
joint work with W. Luo \cite{LZ}, we investigated the LMOV
invariants for resolved conifold which is the large $N$ duality of
the framed unknot \cite{MV}. Moreover, we found that the (reduced)
topological string partition function of $\mathbb{C}^3$ is
equivalent to the Hilbert-Poincare polynomial of certain
cohomological Hall algebra of quiver. Very recently, a series of
works due to D.-E. Diaconescu et al \cite{CDDP,DDP,Dia} showed that
the (refined) GV invariants can be expressed in terms of the Betti
numbers of certain character varieties of algebraic curves based on
the main conjectures in \cite{CHM,HMW}. By the analogues of quiver
varieties and character varieties showed in \cite{HLRV1}, it is
natural to expect there will be an explanation of the integrality of
GV/OV invariants by using quiver varieties. It is also expected that
a general toric Calabi-Yau/quiver variety correspondence may exist
in geometry.
\subsection{Open string model on $(\mathbb{C}^3,\mathcal{D}_\tau)$}
In this short note, we provide a toy model to state this
correspondence through numerical calculations.  More precisely, we
focus on the open topological string on
$(\mathbb{C}^3,\mathcal{D}_\tau)$, where $\mathcal{D}_\tau$ is the
framing $\tau\in \mathbb{Z}$ Aganagic-Vafa A-brane \cite{AV,AKV}.
Its (reduced) topological string partition is given by
\begin{align}
Z^{(\mathbb{C}^3,\mathcal{D}_\tau)}(g_s,\mathbf{x}=(x,0,0,..))=\sum_{n\geq
0}\frac{(-1)^{n(\tau-1)}q^{\frac{n(n-1)}{2}\tau+\frac{n^2}{2}}}{(1-q)(1-q^2)\cdots(1-q^n)}x^n.
\end{align}
We define
\begin{align} \label{partionC3in}
f_{n}^{\tau}(q)=(q^{1/2}-q^{-1/2})[x^n]\text{Log}\left(\sum_{n\geq
0}\frac{(-1)^{n(\tau-1)}q^{\frac{n(n-1)}{2}\tau+\frac{n^2}{2}}}{(1-q)(1-q^2)\cdots(1-q^n)}x^n\right),
\end{align}
where $[x^n]g(x)$ denotes the coefficient of $x^n$ in the series
$g(x)\in \mathbb{Z}[[x]]$ and Log is the plethystic logarithm
introduced in Section \ref{subsection-plethystic}. Applying the work
of Ooguri and Vafa \cite{OV} to this open string model
$(\mathbb{C}^3,\mathcal{D}_\tau)$, we formulate the following
conjecture

\begin{conjecture} \label{Mainconj}
For any $\tau\in \mathbb{Z}$, for a fixed integer $m\geq 1$, we have
\begin{align} \label{integralityC3in}
f_{m}^{\tau}(q)=\sum_{k\in
\mathbb{Z}}N_{m,k}(\tau)q^{\frac{k}{2}}\in \mathbb{Z}[q^{\pm
\frac{1}{2}}].
\end{align}
In other words, for a fixed integer $m\geq 1$, there are only
finitely many $k$, such that the integers $N_{m,k}(\tau)$ are
nonzero.
\end{conjecture}

The rest of this paper is devoted to study the Conjecture
\ref{Mainconj}. We start with the $\tau=0$ case for warming up.
Recall the classical Cauchy identity for Schur functions \cite{Mac},
\begin{align} \label{cauchyidentity}
\sum_{\lambda\in
\mathcal{P}}s_{\lambda}(\mathbf{y})s_{\lambda}(\mathbf{x})=\prod_{i,j\geq
1}\frac{1}{(1-x_iy_j)},
\end{align}
where $\mathbf{x}=(x_1,x_2,...)$, $\mathbf{y}=(y_1,y_2,...)$ and
$\mathcal{P}$ denotes the set of all the partitions. We consider the
specialization $\mathbf{x}=(x,0,0,....)$ and
$\mathbf{y}=q^{\rho}=(q^{-1/2},q^{-3/2},q^{-5/2},...)$, the left
hand side of (\ref{cauchyidentity}) becomes
\begin{align}
\sum_{\lambda\in
\mathcal{P}}s_{\lambda}(q^\rho)s_{\lambda}(\mathbf{x}=(x,0,0,..))=\sum_{n\geq
0}s_n(q^{\rho})x^n=\sum_{n\geq
0}\frac{(-1)^{n}q^{\frac{n^2}{2}}}{(1-q)(1-q^2)\cdots(1-q^n)}x^n,
\end{align}
and the right hand side of (\ref{cauchyidentity}) gives
\begin{align}
\prod_{j\geq 1}(1-xq^{-j+\frac{1}{2}})^{-1}.
\end{align}
Comparing to formulae (\ref{partionC3in}) and
(\ref{integralityC3in}) for when $\tau=0$, by using the definition
of plethystic logarithm Log,  we obtain \makeatletter
\let\@@@alph\@alph
\def\@alph#1{\ifcase#1\or \or $'$\or $''$\fi}\makeatother
\begin{subnumcases}
{N_{m,k}(0)=} 1, & if $m=1$ and $k=0$, \nonumber\\ \nonumber 0, &
otherwise.\
\end{subnumcases}
\makeatletter\let\@alph\@@@alph\makeatother

However, for general $\tau\in \mathbb{Z}$, the Conjecture
\ref{Mainconj} is nontrivial. The first result of this paper is
that, when $\tau \leq 0$, we find the OV invariants $N_{n,k}(\tau)$
can be expressed in terms of the Betti number of certain quiver
variety, which implies the Conjecture \ref{Mainconj} for the case of
$\tau \leq 0$.

\subsection{Proof of the Conjecture \ref{Mainconj} for the case of $\tau\leq 0$}
We construct a quiver of one vertex with $1-\tau$ infinite legs. Let
$\mathcal{Q}_{\tilde{n}(1-\tau)}$ be the associated quiver variety
of the representations in a dimension related to $n$ and $1-\tau$,
we refer to \cite{HLRV2} and Section \ref{Section-quiver} for this
construction. Let $d_{\tilde{n}{(1-\tau)}}=\dim
\mathcal{Q}_{\tilde{n}(1-\tau)}$. There is a Weyl group $S_n$ which
acts on the compactly supported cohomology
$H_c^{1-n+2d_{\tilde{n}{(1-\tau)}}-j}(\mathcal{Q}_{\tilde{n}(1-\tau)};\mathbb{C})$.
Then, we have the following
\begin{theorem} \label{mainthm}
If $n$ is odd
 \makeatletter
\let\@@@alph\@alph
\def\@alph#1{\ifcase#1\or \or $'$\or $''$\fi}\makeatother
\begin{subnumcases}
{N_{n,j}(\tau)=} 0, &$j$ is odd, \label{mainformual1}\\
\nonumber
-(-1)^{(\tau-1)n}\dim(H_c^{1-n+2d_{\tilde{n}{(1-\tau)}}-j}(\mathcal{Q}_{\tilde{n}(1-\tau)};\mathbb{C})^{S_{n}}),
&$j$ is even.
\end{subnumcases}
\makeatletter\let\@alph\@@@alph\makeatother

If $n$ is even
 \makeatletter
\let\@@@alph\@alph
\def\@alph#1{\ifcase#1\or \or $'$\or $''$\fi}\makeatother
\begin{subnumcases}
{N_{n,j}(\tau)=} 0, &$j$ is even, \label{mainformual2}\\
\nonumber
-(-1)^{(\tau-1)n}\dim(H_c^{1-n+2d_{\tilde{n}{(1-\tau)}}-j}(\mathcal{Q}_{\tilde{n}(1-\tau)};\mathbb{C})^{S_{n}}),
&$j$ is odd.
\end{subnumcases}
\makeatletter\let\@alph\@@@alph\makeatother
\end{theorem}
Therefore, as a direct corollary, we have shown
\begin{corollary}
The Conjecture \ref{Mainconj} holds for $\tau\leq 0$.
\end{corollary}

Now the remain question is what about the case of $\tau\geq 1$? We
don't know how to prove this case, but we find it is closely related
to celebrated Rogers-Ramanujan identities (\ref{rr1}) and
(\ref{rr2}).

\subsection{A Rogers-Ramanujan type identity}
Combing(\ref{partionC3in}), (\ref{integralityC3in}) and the
definition of plethystic logarithm Log, Conjecture \ref{Mainconj}
can be rewritten in the form of infinite product
(\ref{infiniteprod}).

Let us take a closer look at the case of $\tau=1$. After some
numerical computations by Maple 13 (see Section \ref{section-RR} for
some of these numerical results), we observe the following rules for
those integers $N_{m,k}(1)$:
\begin{itemize}
\item If $m$ is even, $N_{m,k}(1)\geq 0$, and when $m$ is odd, $N_{m,k}(1)\leq
0$.

\item For a fix integer $m\geq 4$, we define the subset of $\mathbb{Z}$,
\begin{align*}
I_m=\{m+1,m+3,....,m^2-2m-5,m^2-2m-3,(m-1)^2\}\subset \mathbb{Z}.
\end{align*}
then  $N_{m,k}(1)=0$, if $k \in \mathbb{Z}\setminus I_m$. Note that
the last gap in $I_m$ is $(m-1)^2-(m^2-2m-3)=4$. Moreover, we let
$I_1=\{0\}$, $I_2=\{1\}$, $I_3=\{4\}$, according to the computations
in Section \ref{section-RR}.

\end{itemize}
Based on the above observations, let $n_{m,k}=(-1)^{m}N_{m,k}(1)$,
we have the following refined form of the infinite product formula
(\ref{infiniteprod}) for $\tau=1$.
\begin{conjecture}
For a fixed $m\geq 1$, there are only finitely many positive
integers $n_{m,k}$ for $k\in I_m$, such that
\begin{align} \label{deformedRRConjecturein2}
\sum_{n\geq 0} \frac{a^n q^{n^2}}{(1-q)(1-q^2)\cdots
(1-q^n)}=\prod_{m\geq1,l\geq0}\prod_{k\in
\mathbb{Z}}\frac{(1-a^{2m}q^{k+l+2m})^{n_{2m,2k+2m-1}}}{(1-a^{2m-1}q^{k+l+2m-1})^{n_{2m-1,2k+2m-2}}}.
\end{align}
\end{conjecture}

\begin{remark}
After the email correspondence with Ole Warnaar \cite{Wa}, he
suggested the author to rewrite the deformed Rogers-Ramanujan
identity (\ref{deformed-r-r}) into the form
(\ref{deformedRRConjecturein2}) which is related to the
Rogers-Selberg identity \cite{GOW}.
\end{remark}

Note that, (\ref{deformedRRConjecturein2}) can be regarded as a
Rogers-Ramanujan type identity. Recall the two classical
Rogers-Ramanujan identities£º
\begin{align} \label{rr1}
\sum_{n\geq 0}\frac{q^{n^2}}{(1-q)\cdots (1-q^n)}=\prod_{n\geq
0}\frac{1}{(1-q^{5n+1})(1-q^{5n+4})},
\end{align}
\begin{align} \label{rr2}
\sum_{n\geq 0}\frac{q^{n^2+n}}{(1-q)\cdots (1-q^n)}=\prod_{n\geq
0}\frac{1}{(1-q^{5n+2})(1-q^{5n+3})}.
\end{align}
Formulae (\ref{rr1}), (\ref{rr2}) were first discovered by Rogers
\cite{Rogers}, and then rediscovered by Ramanujan \cite{Hardy},
Schur \cite{Schur} and Baxter \cite{Baxter}. Now, there have been
many different proofs and interpretations for them
\cite{And,GM,LW,Bre,Stem}. We refer to \cite{GOW,Wa} for most modern
understanding of the Rogers-Ramanujan identities.

These conjectural integers $n_{m,k}$ appearing in
(\ref{deformedRRConjecturein2}) are important. We expect an explicit
formula for them. Let
\begin{align}
 g_m(q)=\sum_{k\in \mathbb{Z}}n_{m,k}q^k,
\end{align}
  by using Maple 13, we have computed $g_m$ for small $m$ as showed
  in Section \ref{section-RR}.

By our numerical computations, if we let $a=1$ and
$a=q^{\frac{1}{2}}$ respectively in formula
(\ref{deformedRRConjecturein2}), then it recovers the
Rogers-Ramanujan identities (\ref{rr1}) and (\ref{rr2}). Therefore,
(\ref{deformedRRConjecturein2}) can be regarded as an one-parameter
deformed Rogers-Ramanujan identity. From this point of view,
integrality structures of topological string partitions provide a
lot of infinite product formulas, which largely extend the
explorations of Rogers-Ramanujan type formulae.

Finally, in order to give the reader some flavor of these numbers
$n_{m,k}$, we compute the value $f^\tau(1)$ of
(\ref{integralityC3in}) at $q=1$ from Mari\~no-Vafa formula
\cite{MV,LLZ} as follow
\begin{align}
f^{\tau}_{m}(1)=\frac{1}{m^2}\sum_{d|m}\mu(m/d)(-1)^{d\tau}\binom{d(\tau+1)-1}{d-1},
\end{align}
where $\mu(n)$ denotes the M\"obius function. We prove that
\begin{theorem}
For any $m\geq 1$,
\begin{align} \label{integralityftau}
f^{\tau}_{m}(1)=\frac{1}{m^2}\sum_{d|m}\mu(m/d)(-1)^{d\tau}\binom{d(\tau+1)-1}{d-1}
\in \mathbb{Z}.
\end{align}
\end{theorem}
In particular, we obtain
\begin{corollary}
For any $m\geq 1$,
\begin{align}
g_m(1)=\sum_{k}n_{m,k}=(-1)^mf_{m}^{1}(1)\in \mathbb{Z}.
\end{align}
\end{corollary}

The rest of this article is arranged as follows: In Section
\ref{Section-sym}, we introduce the basic notations for partitions,
symmetric functions and plethystic operators. Then, we review the
mathematical structures of topological strings in Section
\ref{Section-top}. We formulate the general Ooguri-Vafa conjecture
by using plethstic operators and we present the explicit form of
Ooguri-Vafa conjecture for the open string model
$(\mathbb{C}^3,\mathcal{D}_\tau)$.  In Section \ref{Section-quiver},
we first review the main results of the work \cite{HLRV2}, as an
application, we prove the Ooguri-Vafa conjecture for
$(\mathbb{C}^3,\mathcal{D}_\tau)$ when $\tau\leq 0$. In Section
\ref{section-RR}, we focus on the Ooguri-Vafa conjecture for
$(\mathbb{C}^3,\mathcal{D}_\tau)$ for the special case $\tau=1$. We
propose the deformed Rogers-Ramanujan type identity
(\ref{deformedRRConjecturein2}). Finally, we present a proof of the
integrality of (\ref{integralityftau}).

\section{Symmetric functions and plethystic operators}
\label{Section-sym}
\subsection{Partitions and symmetric functions}
A partition $\lambda$ is a finite sequence of positive integers $%
(\lambda_1,\lambda_2,..)$ such that $\lambda_1\geq
\lambda_2\geq\cdots$. The length of $\lambda$ is the total number of
parts in $\lambda$ and denoted by
$l(\lambda)$. The weight of $\lambda$ is defined by $|\lambda|=%
\sum_{i=1}^{l(\lambda)}\lambda_i$. If $|\lambda|=d$, we say
$\lambda$ is a partition of $d$ and denoted as $\lambda\vdash d$.
The automorphism group of $\lambda$, denoted by Aut($\lambda$),
contains all the permutations that
permute parts of $\lambda$ by keeping it as a partition. Obviously, Aut($%
\lambda$) has the order $|\text{Aut}(\lambda)|=\prod_{i=1}^{l(\lambda)}m_i(%
\lambda)! $ where $m_i(\lambda)$ denotes the number of times that
$i$ occurs in $\lambda$. Define
$\mathfrak{z}_{\lambda}=|\text{Aut}(\lambda)|\prod_{i=1}^{\lambda}\lambda_i$.

Every partition is identified to a Young diagram. The Young diagram of $%
\lambda$ is a graph with $\lambda_i$ boxes on the $i$-th row for $%
j=1,2,..,l(\lambda)$, where we have enumerated the rows from top to
bottom and the columns from left to right. Given a partition
$\lambda$, we define the conjugate partition $\lambda^t$ whose Young
diagram is the transposed
Young diagram of $\lambda$: the number of boxes on $j$-th column of $%
\lambda^t$ equals to the number of boxes on $j$-th row of $\lambda$, for $%
1\leq j\leq l(\lambda)$. For a box $x=(i,j)\in \lambda$, the hook
length and content are defined to be
$hl(x)=\lambda_i+\lambda_j^{t}-i-j+1$ and $cn(x)=j-i$ respectively.

In the following, we will use the notation $\mathcal{P}_+$ to denote
the set of all the partitions of positive integers. Let $0$ be the
partition of $0$, i.e. the empty partition. Define
$\mathcal{P}=\mathcal{P}_+\cup \{0\}$, and $\mathcal{P}^n$ the $n$
tuple of $\mathcal{P}$.

The power sum symmetric function of infinite variables
$\mathbf{x}=(x_1,..,x_N,..)$ is defined by
$p_{n}(\mathbf{x})=\sum_{i}x_i^n. $ Given a partition $\lambda$, we
define
$p_\lambda(\mathbf{x})=\prod_{j=1}^{l(\lambda)}p_{\lambda_j}(\mathbf{x}).
$ The Schur function $s_{\lambda}(\mathbf{x})$ is determined by the
Frobenius formula
\begin{align}  \label{Frobeniusformula}
s_\lambda(\mathbf{x})=\sum_{\mu}\frac{\chi_{\lambda}(\mu)}{\mathfrak{z}_\mu}p_\mu(\mathbf{x}),
\end{align}
where $\chi_\lambda$ is the character of the irreducible
representation of
the symmetric group $S_{|\lambda|}$ corresponding to $\lambda$, we have $%
\chi_{\lambda}(\mu)=0$ if $|\mu|\neq |\lambda|$. The orthogonality
of character formula gives
\begin{align}  \label{orthog}
\sum_\lambda\frac{\chi_\lambda(\mu)
\chi_\lambda(\nu)}{\mathfrak{z}_\mu}=\delta_{\mu \nu}.
\end{align}

We let $\Lambda(\mathbf{x})$ be the ring of symmetric functions of
$\mathbf{x}=(x_1,x_2,...)$ over the ring $\mathbb{Q}(q,t)$, and let
$\langle \cdot,\cdot \rangle$ be the Hall pair on
$\Lambda(\mathbf{x})$ determined by
\begin{align}
\langle s_\lambda(\mathbf{x}),s_{\mu}(\mathbf{x})
\rangle=\delta_{\lambda,\mu}.
\end{align}
For $\vec{\mathbf{x}}=(\mathbf{x}^1,...,\mathbf{x}^n)$, denote by
$\Lambda(\vec{\mathbf{x}}):=\Lambda(\mathbf{x}^1)\otimes_{\mathbb{Z}}\cdots
\otimes_{\mathbb{Z}}\Lambda(\mathbf{x}^n)$ the ring of functions
separately symmetric in $\mathbf{x}^1,...,\mathbf{x}^n$, where
$\mathbf{x}^i=(x^{i}_{1},x^{i}_{2},...)$. We will study functions in
the ring $\Lambda(\vec{\mathbf{x}})$. For
$\vec{\mu}=(\mu^1,...,\mu^n)\in \mathcal{P}^{n}$, we let
$a_{\vec{\mu}}(\vec{\mathbf{x}})=a_{\mu^1}(\mathbf{x}^1)\cdots
a_{\mu^n}(\mathbf{x}^n)\in\Lambda(\vec{\mathbf{x}})$ be homogeneous
of degree $(|\mu^1|,..,|\mu^n|)$. Moreover, the Hall pair on
$\Lambda(\vec{\mathbf{x}})$ is given by $\langle
a_1(\mathbf{x}^1)\cdots a_n(\mathbf{x}^n), b_1(\mathbf{x}^1)\cdots
b_n(\mathbf{x}^n)\rangle=\langle a_1(\mathbf{x}^1),b_1(\mathbf{x}^1)
\rangle \cdots \langle a_n(\mathbf{x}^n),b_n(\mathbf{x}^n) \rangle$
for $a_1(\mathbf{x}^1)\cdots a_n(\mathbf{x}^n),
b_1(\mathbf{x}^1)\cdots b_n(\mathbf{x}^n)\in
\Lambda(\vec{\mathbf{x}})$.
\subsection{Plethystic operators} \label{subsection-plethystic}
For $d\in \mathbb{Z}_+$, we define the $d$-th Adams operator
$\Psi_d$ as the $\mathbb{Q}$-algebra map on
$\Lambda(\vec{\mathbf{x}})$
\begin{align}
\Psi_{d}(f(\vec{\mathbf{x}};q,t))= f(\vec{\mathbf{x}}^d;q^d,t^d).
\end{align}
Denote by $\Lambda(\vec{\mathbf{x}})^+$ the set of  symmetric
functions with degree $\geq 1$. The plethystic exponential Exp and
logarithm Log are inverse maps
\begin{align}
\text{Exp}: \Lambda(\vec{\mathbf{x}})^+\rightarrow
1+\Lambda(\vec{\mathbf{x}})^+, \ \text{Log}:1+
\Lambda(\vec{\mathbf{x}})^+\rightarrow \Lambda(\vec{\mathbf{x}})^+
\end{align}
respectively defined by (see \cite{HLRV1})
\begin{align}
\text{Exp}(f)=\exp\left(\sum_{d\geq 1}\frac{\Psi_d(f)}{d}\right), \
\text{Log}(f)=\sum_{d\geq 1}\frac{\mu(d)}{d}\Psi_d(\log(f)),
\end{align}
where $\mu$ is the M\"{o}bius function.  It is clear that
\begin{align}
\text{Exp}(f+g)=\text{Exp}(f)\text{Exp}(g), \
\text{Log}(fg)=\text{Log}(f)+\text{Log}(g),
\end{align}
and Exp$(x)=\frac{1}{1-x}$, if we use the expansion
$\log(1-x)=-\sum_{d\geq 1}\frac{x^d}{d}$.

\section{Integrality structures in topological strings}
\label{Section-top}
\subsection{Closed strings and Gopakumar-Vafa conjecture}
Let $X$ be a Calabi-Yau 3-fold, the Gromov-Witten invariants
$K_{g,Q}^{X}$ is the virtual counting of the number of holomorphic
maps $f$ from genus $g$ Riemann suface $C_g$ to $X$ such that
$f_*[C_g]=Q\in H_2(X,\mathbb{Z})$ \cite{HKKPTVVZ}. Define
\begin{align*}
F^X(g_s,\omega)=\sum_{g\geq 0}g_s^{2g-2}F_{g}^X(\omega), \
Z^X(g_s,\omega)=\exp(F^X(g_s,\omega)).
\end{align*}
Usually, the Gromov-Witten invariants $K_{g,Q}^{X}$ are rational
numbers. In 1998,  Gopakumar and Vafa \cite{GV1} conjectured that
the generating function  $F^X(g_s,\omega)$ of Gromov-Witten
invariants can be expressed in terms of integer-valued invariants
$N_{g,Q}^{X}$ as follow
\begin{align} \label{GVformula}
F^X(g_s,\omega)=\sum_{g\geq 0}g_s^{2g-2}\sum_{Q\neq
0}K_{g,Q}^{X}e^{-Q \cdot\omega}=\sum_{g\geq 0, d\geq 1}\sum_{Q\neq
0}\frac{1}{d}N_{g,Q}^{X}\left(2\sin\frac{dg_s}{2}\right)^{2g-2}e^{-dQ\cdot
\omega}.
\end{align}
The invariants $N_{g,Q}^{X}$ are called GV invariants in
literatures. A central question in topological string is how to
define the GV invariants directly. We refer to \cite{HST,KiL,IP,MT}
for some progresses in this direction.

Obviously,  genus 0 part of the Gopakumar-Vafa formula
(\ref{GVformula}) yields
 the multiple covering formula
\cite{AM}:
\begin{align} \label{Multiplecovering}
\sum_{Q\neq 0} K_{0,Q}^{X}e^{-Q\cdot\omega}=\sum_{Q\neq
0}N_{0,Q}^{X}\sum_{d\geq 1}\frac{1}{d^3}e^{-dQ\cdot \omega}.
\end{align}

By using the principle of mirror symmetry, around 1990, Candalas et
al \cite{CDGP} calculated the numbers $N_{0,Q}^{X_5}$ from  formula
(\ref{Multiplecovering}) for quintic Calabi-Yau 3-fold $X_5$, and
found that $N_{0,Q}^{X_5}$ was equal to the number of rational
curves of degree $Q$ in $X_5$ which was hard to compute in
enumerative geometry by classical method. This was the first
important application of the topological string theory in
mathematics.

When $X$ is a toric Calabi-Yau 3-fold which is a toric variety with
trivial canonical bundle \cite{Bou}. Because of its toric symmetry,
the geometric information of a toric Calabi-Yau 3-fold is encoded in
a trivalent graph named ``toric diagram" \cite{AKMV} which is the
gluing of some trivalent vertices. The topological string partition
function $Z^{X}(g_s,\omega)=\exp(F^X(g_s,\omega))$ of a toric
Calabi-Yau 3-fold $X$ can be computed by using the method of
topological vertex \cite{AKMV,LLLZ}. The integrality of the
invariants $N_{g,Q}^{X}$ for toric Calabi-Yau 3-fold $X$ determined
by Gopakumar-Vafa formula (\ref{GVformula}) was  later proved by P.
Peng \cite{Peng} and Konishi \cite{Konishi}.

\subsection{Open strings and Ooguri-Vafa conjecture}
Now we discuss the open topological strings. Let $X$ be a Calabi-Yau
3-fold with a submanifold $\mathcal{D}$, we assume dim
$H_{1}(\mathcal{D},\mathbb{Z})=n$ with basis $\gamma_1,...\gamma_n$.
It is also expected that there are open Gromov-Witten invariants
$K_{\vec{\mu},g,Q}^{(X,\mathcal{D})}$ determined by topological data
$g,\vec{\mu}, Q$, such that $K_{\vec{\mu},g,Q}^{(X,\mathcal{D})}$ is
the virtual counting of holomorphic maps $f$ from genus $g$ Riemann
surface $C_g$ with boundary $\partial C_g$ to $(X,\mathcal{D})$,
such that $f_*([C_g])=Q\in H_2(X,\mathcal{D})$ and $f_*([\partial
C_g])=\sum_{i=1}^n\sum_{j\geq 1}\mu_j^i\gamma_i\in
H_1(\mathcal{D},\mathbb{Z})$. There are no general theory for open
Gromov-Witten invariants, but see \cite{LS,KL} for mathematical
aspects of defining these invariants in special cases.

 The total
free energy and partition function of open topological string on $X$
are defined by as follow
\begin{align} \label{partion-freeenergy}
F^{(X,\mathcal{D})}(\mathbf{x}^1,..,\mathbf{x}^n;g_s,\omega)&=-\sum_{g\geq
0}\sum_{\vec{\mu} \in \mathcal{P}^n\setminus
\{0\}}\frac{\sqrt{-1}^{l(\mu)}}{|Aut(\vec{\mu})|}g_{s}^{2g-2+l(\vec{\mu})}
\sum_{Q\neq 0}K_{\vec{\mu},g,Q}^{(X,\mathcal{D})}e^{-Q\cdot
\omega}\prod_{i=1}^n p_{\mu^i}(\mathbf{x}^i)\\\nonumber
Z^{(X,\mathcal{D})}(\mathbf{x}^1,..,\mathbf{x}^n;g_s,\omega)&=\exp(F^{(X,\mathcal{D})}(\mathbf{x}^1,..,\mathbf{x}^n;g_s,\omega)).
\end{align}
We would like to calculate the partition function
$Z^{(X,\mathcal{D})}(\mathbf{x}^1,..,\mathbf{x}^n;g_s,\omega)$ or
the open Gromov-Witten invariants
$K_{\vec{\mu},g,Q}^{(X,\mathcal{D})}$. For compact Calabi-Yau
3-folds, such as the quintic $X_5$, there are only a few works
devoted to the study of its open Gromov-Witten invariants, for
example, a complete calculation of the disk invariants of $X_5$ with
boundary in a real Lagrangian was given in \cite{PSW}.

Suppose $X$ is a toric Calabi-Yau 3-fold, and $\mathcal{D}$ is a
special Lagrangian submanifold named as Aganagic-Vafa A-brane in the
sense of \cite{AV,AKV}. The open string partition function
$Z^{(X,\mathcal{D})}(\mathbf{x};g_s,\omega)$ can be computed by the
method of topological vertex \cite{AKMV,LLLZ} or topological
recursion developed by Eynard and Orantin \cite{EO1}. The second
approach was first proposed by Mari\~no \cite{Mar}, and studied
further by Bouchard, Klemm, Mari\~no and Pasquetti \cite{BKMP}, the
equivalence of these two methods was proved in \cite{EO2,FLZ}.

The open Gromov-Witten invariants
$K_{\vec{\mu},g,Q}^{(X,\mathcal{D})}$ are rational numbers in
general. Just as in the closed string case \cite{GV1}, the open
topological strings compute the partition function of BPS domain
walls in a related superstring theory \cite{OV}. Ooguri and Vafa
made the prediction that there are integers $N_{\vec{\mu};i,j}$ (OV
invariants) such that
\begin{align}
F^{(X,\mathcal{D})}(\mathbf{x}^1,..,\mathbf{x}^n;g_s,\omega)=
\sum_{d\geq 1}\sum_{\vec{\mu}\in \mathcal{P}^n\setminus
\{0\}}\frac{1}{d}\sum_{i,j}\frac{N_{\vec{\mu},i,j}a^{\frac{di}{2}}q^{\frac{dj}{2}}}{q^{\frac{d}{2}}-q^{-\frac{d}{2}}}
s_{\vec{\mu}}(\vec{\mathbf{x}}),
\end{align}
where $q=e^{\sqrt{-1}g_s}$ and $a=e^{-\omega}$.

Cleanly, one can formulate Ooguri-Vafa conjecture by using the
Plethystic logarithm Log
\begin{conjecture} \label{OVConj}
Let
\begin{align}
f_{\vec{\mu}}(q,a)=(q^{1/2}-q^{-1/2})\langle
\text{Log}(Z^{(X,\mathcal{D})}(\mathbf{x}^1,..,\mathbf{x}^n;q,a)),s_{\vec{\mu}}(\vec{\mathbf{x}})
\rangle,
\end{align}
then we have
\begin{align}
f_{\vec{\mu}}(q,a)=\sum_{i,j}N_{\vec{\mu},i,j}a^{\frac{i}{2}}q^{\frac{j}{2}}
\in \mathbb{Z}[q^{\pm \frac{1}{2}},a^{\pm \frac{1}{2}}].
\end{align}
\end{conjecture}

\begin{remark}
These OV invariants $N_{\vec{\mu},i,j}$ were further refined to be
the invariants $n_{\vec{\mu},g,Q}$  in \cite{LM1,LM2,LMV}. See
\cite{LZ} for a more recent discussion about the LMOV invariants
$n_{\vec{\mu},g,Q}$.
\end{remark}

\subsection{Open string model on $\mathbb{C}^3$}

In this subsection, we focus on the open string model on
$\mathbb{C}^{3}$ with Aganagic-Vafa A-brane $\mathcal{D}_\tau$,
where $\tau\in \mathbb{Z}$ denotes the framing \cite{AV,AKV}. The
topological (open) string partition function of
$(\mathbb{C}^3,\mathcal{D}_\tau)$ is given by the Mari\~no-Vafa
formula \cite{MV} which was proved by \cite{LLZ} and \cite{OP}
respectively:
\begin{align}
Z^{(\mathbb{C}^3,\mathcal{D}_\tau)}(\mathbf{x};q)=\sum_{\lambda\in
\mathcal{P}}\mathcal{H}_{\lambda}(q;\tau) s_{\lambda}(\mathbf{x}),
\end{align}
and where
\begin{align}
\mathcal{H}_\lambda(q;\tau)=(-1)^{|\lambda|\tau}q^{\frac{\kappa_\lambda\tau}{2}}\prod_{x\in
\lambda}\frac{q^{cn(x)/2}}{q^{hl(x)/2}-q^{-hl(x)/2}},
\end{align}
where
$\kappa_{\lambda}=\sum_{i=1}^{l(\lambda)}\lambda_i(\lambda_i-2i+1)$.

The partition function
$Z^{(\mathbb{C}^3,\mathcal{D}_\tau)}(\mathbf{x};q)$ is in fact a
certain generating function of terms which are the coefficients of
highest order of $a$ in the corresponding terms appearing in the
open string partition function of the resolved conifold. That's why
the parameter $a$ does not appear in the expression
$Z^{(\mathbb{C}^3,\mathcal{D}_\tau)}(\mathbf{x};q)$. We refer to
\cite{LZ} for more details.

Applying the Ooguri-Vafa Conjecture \ref{OVConj} to
$Z^{(\mathbb{C}^3,\mathcal{D}_\tau)}(g_s,\mathbf{x})$, it follows
that for any $\tau\in \mathbb{Z}$ and $\mu \in \mathcal{P}_{+}$, we
have
\begin{align} \label{OVConjC3}
f_{\mu}^{\tau}(q)=(q^{\frac{1}{2}}-q^{-\frac{1}{2}})\langle
\text{Log}(Z^{(\mathbb{C}^3,\mathcal{D}_\tau)}(\mathbf{x};q)),s_{\mu}(\mathbf{x})
\rangle \in \mathbb{Z}[q^{\pm \frac{1}{2}}].
\end{align}

In particular, if we let $\mathbf{x}=(x,0,0,...)$, then
\begin{align} \label{partionC3}
Z^{(\mathbb{C}^3,\mathcal{D}_\tau)}(g_s,\mathbf{x}=(x,0,0,..))&=\sum_{n\geq
0}\mathcal{H}_{(n)}(q;\tau)x^n\\\nonumber
 &=\sum_{n\geq
0}\frac{(-1)^{n(\tau-1)}q^{\frac{n(n-1)}{2}\tau+\frac{n^2}{2}}}{(1-q)(1-q^2)\cdots(1-q^n)}x^n,
\end{align}
and
\begin{align} \label{fntau1}
f_{n}^{\tau}(q):=f_{(n)}^{\tau}(q)=(q^{\frac{1}{2}}-q^{-\frac{1}{2}})[x^n]\text{Log}\left(\sum_{n\geq
0}\frac{(-1)^{n(\tau-1)}q^{\frac{n(n-1)}{2}\tau+\frac{n^2}{2}}}{(1-q)(1-q^2)\cdots(1-q^n)}x^n\right)
\end{align}
Therefore, formula (\ref{OVConjC3}) implies that for any $\tau \in
\mathbb{Z}$ and $n\in \mathbb{Z}_{\geq 1}$,
\begin{align}
f_{n}^{\tau}(q)\in \mathbb{Z}[q^{\pm \frac{1}{2}}].
\end{align}
Therefore, we formulate the Ooguri-Vafa conjecture for
$(\mathbb{C}^3,\mathcal{D}_\tau)$ as follow
\begin{conjecture} \label{Mainconj-section}
For any $\tau\in \mathbb{Z}$, for a fixed integer $n\geq 1$, we have
\begin{align} \label{fntau2}
f_{n}^{\tau}(q)=\sum_{k\in
\mathbb{Z}}N_{n,k}(\tau)q^{\frac{k}{2}}\in \mathbb{Z}[q^{\pm
\frac{1}{2}}].
\end{align}
In other words, for a fixed integer $n\geq 1$, there are only
finitely many $k$, such that the integers $N_{n,k}(\tau)$ are
nonzero.
\end{conjecture}

Now identity (\ref{fntau1}) is equivalent to
\begin{align}
\text{Log}\left(\sum_{n\geq
0}\frac{(-1)^{n(\tau-1)}q^{\frac{n(n-1)}{2}\tau+\frac{n^2}{2}}}{(1-q)(1-q^2)\cdots(1-q^n)}x^n\right)=\sum_{n\geq
0}\sum_{k\in
\mathbb{Z}}\frac{N_{n,k}(\tau)q^{\frac{k}{2}}}{(q^{\frac{1}{2}}-q^{-\frac{1}{2}})}x^n.
\end{align}

By using the properties of plethystic operators introduced in
Section \ref{subsection-plethystic}, we can write (\ref{partionC3})
in the form of infinite product as follow:
\begin{align} \label{infiniteprod}
\sum_{n\geq
0}\frac{(-1)^{n(\tau-1)}q^{\frac{n(n-1)}{2}\tau+\frac{n^2}{2}}}{(1-q)(1-q^2)\cdots(1-q^n)}x^n&=\text{Exp}\left(
\sum_{n\geq 0}\sum_{k\in
\mathbb{Z}}\frac{N_{n,k}(\tau)q^{\frac{k}{2}}}{(q^{\frac{1}{2}}-q^{-\frac{1}{2}})}x^n\right)\\\nonumber
&=\text{Exp}\left(- \sum_{n\geq 0}\sum_{k\in \mathbb{Z}}\sum_{l\geq
0}N_{n,k}(\tau)q^{\frac{k}{2}}q^{\frac{1}{2}+l}x^n\right)\\\nonumber
&=\prod_{n\geq 0}\prod_{k\in \mathbb{Z}}\prod_{l\geq
0}\left(1-q^{\frac{k}{2}}q^{\frac{1}{2}+l}x^n\right)^{N_{n,k}(\tau)},
\end{align}
where we have used the formal series
\begin{align} \label{formalexpansion}
\frac{1}{1-q}=1+q+q^2+\cdots.
\end{align}
\begin{remark}
When we write the formula (\ref{partionC3}) into the form of
infinite product, one can also use the formal series
\begin{align}
\frac{1}{1-q^{-1}}=1+q^{-1}+q^{-2}+\cdots.
\end{align}
In order to make the connections to the Rogers-Ramanujan identities,
here we need the infinite product form (\ref{infiniteprod}) by using
the expansion (\ref{formalexpansion}).
\end{remark}

In the following two sections, we will show that when $\tau\in
\mathbb{Z}$ and $\tau\leq 0$, these integers $N_{n,k}(\tau)$ can be
interpreted as the Betti numbers of certain cohomologies of quiver
varieties, which finishes the proof of Conjecture
\ref{Mainconj-section} for $\tau\leq 0$. As to the case of $\tau\geq
1$, we study carefully for the special case of  $\tau=1$, and we
find that these integers $N_{n,k}(1)$ together with formula
(\ref{infiniteprod}) give a deformed version of the famous
Rogers-Ramanujan identities (\ref{rr1}) and (\ref{rr2}).

\section{Cohomologies of quiver varieties} \label{Section-quiver}
Motivated by the previous works in gauge theory \cite{Kr,KN}, H.
Nakajima \cite{N1,N2} introduced the quiver varieties and
illustrated how to use them to construct the geometric
representations of Kac-Moody algebras. From then on, quiver
varieties became to be the central objects in mathematics, we refer
to \cite{Kir} for the introduction to quiver varieties.  Quiver
varieties have a lot of structures and applications, for example,
they can be used to prove the famous Kac's conjectures \cite{Kac}.
\subsection{Kac's conjecture}
We follow the notations in \cite{HLRV1,HLRV2}. Take a ground field
$\mathbb{K}$, denote by $\Gamma=(I,\Omega)$ a quiver with
$I=\{1,...,n\}$ the set of vertices, and $\Omega$ the set of edges
of $\Gamma$. For $\gamma\in \Omega$, let $h(\gamma), t(\gamma)\in I$
denote the head and tail of $\gamma$. A representation of $\Gamma$
of dimension $\mathbf{v}=\{v_i\}_{i\in I}\in (\mathbb{Z}_{\geq
0})^{n}$ over $\mathbb{K}$ is a collection of $\mathbb{K}$-linear
maps $\phi_{\gamma}: \mathbb{K}^{v_{t(\gamma)}}\rightarrow
\mathbb{K}^{v_{h(\gamma)}}$ for each $\gamma\in \Omega$ that can be
identified with matrices by using the canonical base of
$\mathbb{K}^m$. A representation is said to be absolutely
indecomposable over $\mathbb{K}$, if it is nontrivial and not
isomorphic to a direct sum of two nontrivial representations of
$\Gamma$ over $\mathbb{K}$. A indecomposable representation is said
to be absolutely indecomposable over $\mathbb{K}$, if it is still
indecomposable over any extension field of $\mathbb{K}$.

In order to study the representation theory of general quiver
$\Gamma$,  Kac \cite{Kac} introduced $A_{\mathbf{v}}(q)$, the number
of isomorphic classes of absolutely indecomposable representations
of $\Gamma$ with dimension $\mathbf{v}=(v_1,...,v_n)$ over finite
field $\mathbb{F}_q$. Following the idea of \cite{Kac}, J. Hua
firstly computed the Kac polynomial $A_{\mathbf{v}}(q)$ in the
following form:
\begin{align} \label{Huaformula}
&\sum_{\mathbf{v}\in \mathbb{Z}_{\geq 0}^n\
\{0\}}A_{\mathbf{v}}(q)\prod_{i=1}^nT_i^{v_i}=(q-1)\\\nonumber \cdot
&\text{Log}\left(\sum_{(\pi^1,...,\pi^n)\in
\mathcal{P}^n}\frac{\prod_{\gamma\in \Omega}q^{\langle
\pi^{t(\gamma)},\pi^{h(\gamma)}\rangle}}{\prod_{i\in I}q^{\langle
\pi^i,\pi^i\rangle}\prod_{k\geq
1}\prod_{j=1}^{m_k(\pi^i)}(1-q^{-j})}\prod_{i=1}^{n}T_i^{|\pi^i|}\right)
\end{align}
where $\langle , \rangle$ is the pairing on partitions defined by
\begin{align}
\langle \lambda, \mu
\rangle=\sum_{i,j}\text{min}(i,j)m_i(\lambda)m_j(\mu).
\end{align}
Kac \cite{Kac} proved that $A_{\mathbf{v}}(q)$ has integer
coefficients and made two remarkable conjectures:

(i) If $\Gamma$ has no edge-loops, then the constant term of
$A_{\mathbf{v}}(0)$ is equal to the multiplicity of the root
$\mathbf{v}$ in the corresponding Kac-Moody algebra
$\mathfrak{g}(\Gamma)$.

(ii) The Kac polynomial $A_{\mathbf{v}}(q)$ has nonnegative
coefficients.

Conjecture (i) was proved by Hausel \cite{Hau} and Conjecture (ii)
was completely settled  by T. Hausel, E. Letellier and F.
Rodriguez-Villegas \cite{HLRV2} by using the theory of Nakajima
quiver varieties and computing via arithmetic Fourier transform.
They introduced the following function which largely generalizes
Hua's formula (\ref{Huaformula})
\begin{align}
&\mathbb{H}(\mathbf{x}^1,...,\mathbf{x}^n;q):=(q-1)\\\nonumber \cdot
&\text{Log}\left(\sum_{(\pi^1,...,\pi^n)\in
\mathcal{P}^n}\frac{\prod_{\gamma\in \Omega}q^{\langle
\pi^{t(\gamma)},\pi^{h(\gamma)}\rangle}}{\prod_{i\in I}q^{\langle
\pi^i,\pi^i\rangle}\prod_{k\geq
1}\prod_{j=1}^{m_k(\pi^i)}(1-q^{-j})}\prod_{i=1}^{n}\tilde{H}_{\pi^i}(\mathbf{x}^i;q)\right),
\end{align}
where $\tilde{H}_{\pi^i}(\mathbf{x}^i;q)$ is the (transformed)
Hall-Littlewood polynomial introduced in \cite{GH}.

For $s_{\vec{\mu}}(\vec{\mathbf{x}}):=s_{\mu^1}(\mathbf{x}^1)\cdots
s_{\mu^n}(\mathbf{x}^n)$, we let
\begin{align}
\mathbb{H}_{\vec{\mu}}^s(q):=\langle\mathbb{H}(\mathbf{x}^1,..,\mathbf{x}^n;q),s_{\vec{\mu}}(\vec{\mathbf{x}})
\rangle.
\end{align}

\subsection{The quiver varieties $\mathcal{Q}_{\tilde{\mathbf{v}}}$}
In their remarkable work \cite{HLRV2},  Hausel et al found the
geometric interpretation of $\mathbb{H}_{\vec{\mu}}^s(q)$ by
computing, via arithmetic Fourier transform,  the dimension of
certain cohomologies of Nakajima quiver varieties. Let us briefly
recall the main results in \cite{HLRV2}.

 We denote the space of all the representations of
$\Gamma$ over $\mathbb{K}$ with dimension $\mathbf{v}$  by
\begin{align}
\text{Rep}_{\mathbb{K}}(\Gamma,\mathbf{v}):=\bigoplus_{\gamma\in
\Omega}\text{Mat}_{v_{h(\gamma)},v_{t(\gamma)}}(\mathbb{K}).
\end{align}
Let $\text{GL}_{\mathbf{v}}=\prod_{i\in
I}\text{GL}_{v_i}(\mathbb{K})$ and
$\mathfrak{gl}_{\mathbf{v}}=\prod_{i\in
I}\mathfrak{gl}_{v_i}(\mathbb{K})$. The algebraic group
$\text{GL}_{\mathbf{v}}$ acts on
$\text{Rep}_{\mathbb{K}}(\Gamma,\mathbf{v})$ as
\begin{align}
(g\cdot \phi)_{\gamma}=g_{v_{h(\gamma)}}
\phi_{\gamma}g_{v_{t(\gamma)}}^{-1}
\end{align}
for any $g=(g_i)_{i\in I}\in \text{GL}_{\mathbf{v}}$,
$\phi=(\phi_\gamma)_{\gamma\in \Omega}\in
\text{Rep}_{\mathbb{K}}(\Gamma,\mathbf{v})$. Since the diagonal
center $(\lambda I_{v_i})_{i\in I}\in \text{GL}_{\mathbf{v}}$ acts
trivially on $\text{Rep}_{\mathbb{K}}(\Gamma,\mathbf{v})$, the
action reduces to an action of
$\text{G}_\mathbf{v}=\text{GL}_{\mathbf{v}}/\mathbb{K}^{\times}$.

Let $\overline{\Gamma}$ be the double quiver of $\Gamma$, namely,
$\overline{\Gamma}$ has the same vertices as $\Gamma$, but the set
of edges are given by
$\overline{\Omega}:=\{\gamma,\gamma^*|\gamma\in \Omega\}$, where
$h(\gamma^*)=t(\gamma)$ and $t(\gamma^*)=h(\gamma)$. By the trace
pairing, we may identify
$\text{Rep}_{\mathbb{K}}(\overline{\Gamma},\mathbf{v})$ with the
cotangent bundle
$\text{T}^*\text{Rep}_{\mathbb{K}}(\Gamma,\mathbf{v})$. We define
the moment map
\begin{align}
\mu_{\mathbf{v}}:
&\text{Rep}_{\mathbb{K}}(\overline{\Gamma},\mathbf{v})\rightarrow
\mathfrak{gl}_{\mathbf{v}}^0\\\nonumber &(x_{\gamma})_{\gamma\in
\overline{\Omega}} \mapsto \sum_{\gamma\in
\Omega}[x_{\gamma},x_{\gamma^*}].
\end{align}
where $\mathfrak{gl}_{\mathbf{v}}^0=\{(X_i)_{i\in I}\in
\mathfrak{gl}_\mathbf{v}|\sum_{i\in I}\text{Tr}(X_i)=0\}$ is
identified with the dual of the Lie algebra of
$\text{G}_{\mathbf{v}}$. It is a $\text{G}_{\mathbf{v}}$-equivariant
map. For $\mathbf{\xi}=(\xi_i)_{i\in I}\in \mathbb{K}^I$ such that
$\xi\cdot \mathbf{v}=\sum_{i}\xi_i v_i=0$, then
\begin{align}
(\xi_i I_{v_i})_{i\in I} \in \mathfrak{gl}_{\mathbf{v}}^0.
\end{align}
For  such a $\xi\in \mathbb{K}^I$, the affine variety
$\mu_{\mathbf{v}}^{-1}(\xi)$ inherits a
$\text{G}_{\mathbf{v}}$-action. The quiver variety
$\mathcal{Q}_{\mathbf{v}}$ is the affine GIT quotient
\begin{align}
\mu_{\mathbf{v}}^{-1}(\xi)//\text{G}_{\mathbf{v}}.
\end{align}
The (related) quiver varieties were studied by many authors in past
two decades, for example \cite{N1,N2,Lus,CB}.

Let $\tilde{\Gamma}_{\mathbf{v}}$ on vertex set
$\tilde{I}_{\mathbf{v}}$ be the quiver obtained from
$(\Gamma,\mathbf{v})$ by adding at each vertex $i\in I$ a leg of
length $v_i-1$ with all the edges oriented towards the vertex $i$.
Let $\tilde{\mathbf{v}}\in \mathbb{Z}_{\geq
0}^{\tilde{I}_{\mathbf{v}}}$ be the dimension vector with coordinate
$v_i$ at $i\in I\subset \tilde{I}_\mathbf{v}$ and with coordinates
$(v_i-1,v_i-2,...,1)$ on the leg attached to the vertex $i\in I$. We
let $\mathcal{Q}_{\tilde{\mathbf{v}}}$ be the quiver variety
attached to the quiver $\tilde{\Gamma}_{\mathbf{v}}$ with parameter
$\tilde{\xi}$ such that $\tilde{\xi}\cdot \tilde{\mathbf{v}}=0$.
Denote by $\tilde{C}_{\mathbf{v}}$ the Cartan matrix of the quiver
$\tilde{\Gamma}_{\mathbf{v}}$, then
\begin{align}
d_{\tilde{\mathbf{v}}}:=1-\frac{1}{2}\tilde{\mathbf{v}}^t\tilde{C}_{\mathbf{v}}\tilde{\mathbf{v}}
\end{align}
equals $\frac{1}{2}\dim \mathcal{Q}_{\tilde{\mathbf{v}}} $ if
$\mathcal{Q}_{\tilde{\mathbf{v}}}$ is nonempty.

Let $W_{\mathbf{v}}:=S_{v_1}\times\cdots\times S_{v_n}$ be the Weyl
group of the group $GL_{\mathbf{v}}:=GL_{v_1}\times \cdots \times
GL_{v_n}$, it acts on
$H_c^{*}(\mathcal{Q}_{\tilde{\mathbf{v}}};\mathbb{C})$ by the work
of Nakajima \cite{N1,N2}. We denote by
$\chi^{\vec{\mu}}=\chi^{\mu^1}\cdots \chi^{\mu^n}:
W_{\mathbf{v}}\rightarrow \mathbb{C}^{\times}$ the exterior product
of the irreducible characters $\chi^{\mu^i}$ of the symmetric group
$S_{v_i}$ in the notation of \cite{Mac}. In particular,
$\chi^{(v_i)}$ is the trivial character and $\chi^{(1^{v_i})}$ is
the sign character $\epsilon_i: S_{v_i}\rightarrow \{\pm 1\}$.

The main result of \cite{HLRV2} is
\begin{theorem}[Theorem 1.4 and Corollary 1.5 in \cite{HLRV2}]
We have
\begin{align}
\mathbb{H}_{\vec{\mu}}^{s}(q)=\sum_{i}\langle\rho^{2i},\epsilon
\chi^{\vec{\mu}}\rangle_{W_\mathbf{v}}q^{i-d_{\tilde{\mathbf{v}}}},
\end{align}
where $\langle\rho^{2i},\epsilon
\chi^{\vec{\mu}}\rangle_{W_\mathbf{v}}$ is the multiplicity of
$\epsilon \chi^{\vec{\mu}}$ in the representation $\rho^{2i}$ of
$W_\mathbf{v}$ in
$H_c^{2i}(\mathcal{Q}_{\tilde{\mathbf{v}}};\mathbb{C})$.

In particular, for
$\vec{\mu}=(1^{\mathbf{v}})=((1^{v_1}),...,1^{v_n})\in
\mathcal{P}^n$, we have
\begin{align}
\mathbb{H}_{1^{\mathbf{v}}}^{s}(q)=\sum_{j}\dim(H_c^{2j}(\mathcal{Q}_{\tilde{\mathbf{v}}};\mathbb{C})^{W_{\mathbf{v}}})q^{j-d_{\tilde{\mathbf{v}}}}
\end{align}

\end{theorem}

\subsection{A special case}
For the quiver $\Gamma=(I,\Omega)$, we attach $k_i\geq 1$ infinite
legs to  each vertex $i\in I$ of $\Gamma$. Let
$\mathbf{k}=(k_1,..,k_n)\in \mathbb{Z}_{\geq 1}^n$.  We set all the
arrows on the new legs point towards the vertex. Given a dimension
vector $\mathbf{v}\in \mathbb{Z}_{\geq 0}^n\setminus \{0\}$, one can
also construct a quiver varieties similarly including the previous
construction as the special case of all $k_i=1$. More precisely, let
$\tilde{\Gamma}(\mathbf{k})$  be the quiver obtained from
$(\Gamma,\mathbf{v})$ by adding at each vertex $i\in I$ $k_i$
infinite legs of
 the edges all oriented toward the vertex $i$. Denote by $\tilde{\mathbf{v}}(\mathbf{k})$
 the dimension vector with coordinate $v_i$ at $i\in I$ and with the same coordinates $(v_i-1,v_i-2,..,1,0,0,..)$ on
 the $k_i$ legs attached to the vertex $i\in I$. Now we let
 $\mathcal{Q}_{\tilde{\mathbf{v}}(\mathbf{k})}$ be the quiver variety  associated to quiver
 $\tilde{\Gamma}(\mathbf{k})$.

\begin{corollary} [\cite{HLRV2}, Proposition 3.4, by changing $q\rightarrow q^{-1}$]
We have the identity
\begin{align}
(q^{-1}-1)&\text{Log}\left(\sum_{\mathbf{v}\in \mathbb{Z}_{\geq
0}^n}\frac{q^{\frac{1}{2}(\gamma(\mathbf{v}(\mathbf{k}))+\delta(\mathbf{v}(\mathbf{k})))}}{\prod_{i=1}^n(1-q)\cdots
(1-q^{v_i})}(-1)^{\delta(\mathbf{v}(\mathbf{k}))}\prod_{i=1}^nT_i^{v_i}\right)\\\nonumber
&=\sum_{\mathbf{v}\in \mathbb{Z}_{\geq
0}^n}\mathbb{H}_{1^{\mathbf{v}(\mathbf{k})}}^s(q^{-1})(-1)^{\delta(\mathbf{v}(\mathbf{k}))}\prod_{i=1}^nT_i^{v_i}.
\end{align}
where
\begin{align} \label{HLRVformulaspecial}
\gamma(\mathbf{v}(\mathbf{k}))=\sum_{i=1}^{n}(2-k_i)v_i^2-2\sum_{\gamma
\in \Omega}v_{t(\gamma)}v_{h(\gamma)}, \
\delta(\mathbf{v}(\mathbf{k}))= \sum_{i=1}^nk_iv_i
\end{align}
and
$(1^{\mathbf{v}(\mathbf{k})})=((1^{v_1})^{k_1},...,(1^{k_n})^{k_n})$
where $(1^{v_i})^{k_i}$ denotes that $(1^{v_i})$ appears $k_i$
times.
\end{corollary}
By Theorem 4.1, we obtain
\begin{align} \label{quivercoh}
\mathbb{H}_{1^{\mathbf{v}(\mathbf{k})}}^{s}(q^{-1})=\sum_{j}
\dim(H_c^{2j}(\mathcal{Q}_{\tilde{\mathbf{v}}(\mathbf{k})};\mathbb{C})^{W_{\mathbf{v}}})q^{d_{\tilde{\mathbf{v}}{(\mathbf{k})}}-j}.
\end{align}

Now, we can finish the proof of Theorem \ref{mainthm}.

\begin{proof}

For the framing $\tau\in \mathbb{Z}_{\leq 0}$, we take $k=1-\tau\in
\mathbb{Z}_{\geq 1}$.   Consider the one vertex quiver
$\Gamma=\bullet$, we construct a new quiver $\Gamma(k)$ with the
unique vertex attached with $k$ infinite legs as showed above.
Associate a dimension vector  $\tilde{n}(k)$ to quiver $\Gamma(k)$,
we have the quiver variety $\mathcal{Q}_{\tilde{n}(k)}$ by the
construction showed previously. Now combining formulae
(\ref{HLRVformulaspecial}) and (\ref{quivercoh}) together in this
special case,  by the variable change $x=q^{1/2}T$, we obtain

\begin{align}
&(q^{1/2}-q^{-1/2})\text{Log}\left(\sum_{n\geq
0}\frac{(-1)^{n(\tau-1)}q^{\frac{n(n-1)}{2}\tau+\frac{n^2}{2}}}{(1-q)(1-q^2)\cdots(1-q^n)}x^n\right)\\\nonumber
&=-\sum_{n\geq
0}\mathbb{H}^s_{1^{n(k)}}(q)^{1/2-n/2}(-1)^{(\tau-1)n}x^n\\\nonumber
&=-\sum_{n\geq 0}\sum_{j}
\dim(H_c^{2j}(\mathcal{Q}_{\tilde{n}(k)};\mathbb{C})^{S_{n}})q^{\frac{1-n}{2}+d_{\tilde{n}{(k)}}-j}
(-1)^{(\tau-1)n}x^n
\end{align}

Therefore,
\begin{align}
f_{n}^{\tau}(q)=-\sum_{j}
\dim(H_c^{2j}(\mathcal{Q}_{\tilde{n}(1-\tau)};\mathbb{C})^{S_{n}})q^{\frac{1-n}{2}+d_{\tilde{n}{(1-\tau)}}-j}
(-1)^{(\tau-1)n}
\end{align}

Comparing to the formula (\ref{fntau2}), we obtain the formulae
(\ref{mainformual1}) and (\ref{mainformual2}) in Theorem
\ref{mainthm}.

\end{proof}

\section{Deformed Rogers-Ramanujan identities} \label{section-RR}
In the above section, we have interpreted and proved the integrality
of Ooguri-Vafa invariants $N_{n,j}(\tau)$ for $\tau\in \mathbb{Z}$
and $\tau\leq 0$. It is natural to ask how about the case $\tau \geq
1$?

As shown in Section \ref{Section-top}, Conjecture
\ref{Mainconj-section} leads to the conjectural infinite product
formula (\ref{infiniteprod}).

In the following, we study carefully  for the formula
(\ref{infiniteprod}) in $\tau=1$. We find that if we let
$n_{m,k}:=(-1)^{m}N_{m,k}$, then $n_{m,k}$ will be nonnegative.
After some concrete computation by using Maple 13,  we propose
\begin{conjecture} \label{deformedRRConj}
For a fixed integer $m\geq 1$, there exist finite many positive
integers $n_{m,k}$,  such that
\begin{align} \label{deformed-r-r}
\sum_{n\geq 0}\frac{q^{n^2}}{(1-q)\cdots
(1-q^n)}(q^{-\frac{1}{2}}x)^n=\prod_{m\geq 1}\prod_{k\in
\mathbb{Z}}\prod_{l\geq
0}\left(1-q^{\frac{k+1}{2}+l}x^m\right)^{(-1)^m n_{m,k}}
\end{align}
In particular, when $x=q^{\frac{1}{2}}$ and $x=q^{\frac{3}{2}}$,
these integers $n_{m,k}$ together with the formula
(\ref{deformed-r-r}) yield the two Rogers-Ramanujan identities
(\ref{rr1}) and (\ref{rr2}).
\end{conjecture}

Let us give some numerical checks for Conjecture
\ref{deformedRRConj}. We introduce the polynomial
\begin{align}
g_m(q)=\sum_{k\in \mathbb{Z}}n_{m,k}q^k,
\end{align}
By using Maple 13, we have computed the polynomial $g_m(q)$ for
$1\leq m\leq 18$. Here is a list for them when $m\leq 6$:
\begin{align*}
g_{1}(q)&=1 \\\nonumber g_2(q)&=q,\\\nonumber
g_3(q)&=q^4,\\\nonumber g_4(q)&=q^5+q^9,  \\\nonumber
g_5(q)&=q^6+q^8+q^{10}+q^{12}+q^{16},\\\nonumber
g_6(q)&=q^7+2q^9+q^{11}+3q^{13}+q^{15}+2q^{17}+q^{19}+q^{21}+q^{25}
\end{align*}
If we let $x=q^{\frac{1}{2}}$, identity (\ref{deformed-r-r}) becomes
\begin{align} \label{r-r-1}
\sum_{n\geq 0}\frac{q^{n^2}}{(1-q)\cdots (1-q^n)}=\prod_{m\geq
1}\prod_{k\geq 0}\prod_{l\geq
0}\left(1-q^{\frac{m+k+1}{2}+l}\right)^{(-1)^m
n_{m,k}}=\prod_{i}\prod_{l\geq 0}(1-q^{i+l})^{n_i}
\end{align}
where $n_i=\sum_{m+k+1=2i}(-1)^mn_{m,k}$, our computations imply
that: \makeatletter
\let\@@@alph\@alph
\def\@alph#1{\ifcase#1\or \or $'$\or $''$\fi}\makeatother
\begin{subnumcases}
{n_i=}  \nonumber -1, &$i=5k+1$ or $5k+4$, for $k\geq 0$\\\nonumber
1, & $i=5k+2$ or $5k+5$, for $k\geq 0$ \\ 0 \nonumber, & otherwise,
\end{subnumcases}
\makeatletter\let\@alph\@@@alph\makeatother It turns out formula
(\ref{r-r-1}) gives the first Rogers-Ramanujan identity (\ref{rr1}).

Similarly, letting $x=q^{\frac{3}{2}}$, identity
(\ref{deformed-r-r}) becomes
\begin{align} \label{r-r-2}
\sum_{n\geq 0}\frac{q^{n^2+n}}{(1-q)\cdots (1-q^n)}=\prod_{m\geq
1}\prod_{k\in \mathbb{Z}}\prod_{l\geq
0}\left(1-q^{\frac{3m+k+1}{2}+l}\right)^{(-1)^m
n_{m,k}}=\prod_{i}\prod_{l\geq 0}(1-q^{i+l})^{r_i}
\end{align}
where $r_i=\sum_{3m+k+1=2i}(-1)^mn_{m,k}$, we find that:
\makeatletter
\let\@@@alph\@alph
\def\@alph#1{\ifcase#1\or \or $'$\or $''$\fi}\makeatother
\begin{subnumcases}
{r_i=}  \nonumber -1, &$i=5k+2$, for $k\geq 0$\\\nonumber 1, &
$i=5k+4$, for $k\geq 0$ \\0,\nonumber & otherwise,
\end{subnumcases}
\makeatletter\let\@alph\@@@alph\makeatother Hence formula
(\ref{r-r-2}) gives the second Rogers-Ramanujan identity
(\ref{rr2}).

Formula (\ref{deformed-r-r}) in Conjecture 5.1 is a formula  of type
``infinite sum=infinite product".  We expect it could be interpreted
by the denominator formula for some kinds of root system
\cite{Kac2}.

These conjectural integers $n_{m,k}$ appearing in Conjecture
\ref{deformedRRConj} are important. Although we have not obtained an
explicit formula for them, we have an explicit formula for the value
of $\sum_{k}n_{m,k}$ for any $m\geq 1$.

First, recall the definition of  $f^{\tau}_{\mu}(q)$ in
(\ref{OVConjC3}), we have
\begin{align} \label{ZC3}
Z^{(\mathbb{C}^3,\mathcal{D}_\tau)}(\mathbf{x};q)=\text{Exp}\left(\frac{1}{q^{\frac{1}{2}}-q^{-\frac{1}{2}}}
\sum_{\mu\in
\mathcal{P}_+}f_{\mu}^{\tau}(q)s_{\mu}(\mathbf{x})\right).
\end{align}
By the formula (\ref{partion-freeenergy}) for the case of
$(\mathbb{C}^3,\mathcal{D}_\tau)$, the open string free energy is
given by
\begin{align} \label{FC3}
F^{(\mathbb{C}^3,\mathcal{D}_\tau)}(\mathbf{x})=\log
Z^{(\mathbb{C}^3,\mathcal{D}_\tau)}(\mathbf{x};q)=-\sum_{\mu\in
\mathcal{P}_+}\sum_{g\geq
0}\frac{\sqrt{-1}^{l(\mu)}}{|Aut(\mu)|}K_{\mu,g,\frac{|\mu|}{2}}^{(\mathbb{C}^3,\mathcal{D}_\tau)}g_s^{2g-2+l(\mu)}p_{\mu}(\mathbf{x}),
\end{align}
which is the generating function the Gromov-Witten invariants
$K_{\mu,g,\frac{|\mu|}{2}}^{(\mathbb{C}^3,\mathcal{D}_\tau)}$, and
where $q=e^{ig_s}$.

An explicit expression for
$K_{\mu,g,\frac{|\mu|}{2}}^{(\mathbb{C}^3,\mathcal{D}_\tau)}$ is
obtained in \cite{KL}  (cf. formula (15) in \cite{LZ}). In
particular, for $m\geq 1$ and $g=0$, we have
\begin{align}
K_{m,0,\frac{m}{2}}^{(\mathbb{C}^3,\mathcal{D}_\tau)}=\frac{(-1)^{m\tau}}{m^2}\binom{m(\tau+1)-1}{m-1}.
\end{align}

Then, we combine (\ref{ZC3}) and (\ref{FC3}) together, and consider
the specialization $\mathbf{x}=(x,0,0,...)$, it follows that
\begin{align} \label{GWOVformula}
\sum_{m\geq 1}\sum_{g\geq
0}K_{m,g,\frac{m}{2}}^{(\mathbb{C}^3,\mathcal{D}_\tau)}g_{s}^{2g}x^m=\sum_{d\geq
1}\frac{\sqrt{-1}g_s}{d(q^{\frac{d}{2}}-q^{-\frac{d}{2}})}\sum_{m\geq
1}\sum_{k}N_{m,k}(\tau)q^{\frac{dk}{2}}x^{dm}.
\end{align}

Taking the coefficients of $x^m$ in (\ref{GWOVformula}), and
considering the limit $g_s\rightarrow 0$, we obtain
\begin{align} \label{Km0}
K_{m,0,\frac{m}{2}}^{(\mathbb{C}^3,\mathcal{D}_\tau)}=\sum_{d|m}\frac{1}{d^2}\sum_{k}N_{m/d,k}(\tau),
\end{align}
where we have used
\begin{align}
\lim_{g_s\rightarrow
0}\frac{\sqrt{-1}g_s}{q^{\frac{d}{2}}-q^{-\frac{d}{2}}}=\lim_{g_s\rightarrow
0}\frac{\sqrt{-1}g_s}{e^{\sqrt{-1}g_s\frac{d}{2}}-e^{-\sqrt{-1}g_s\frac{d}{2}}}=\frac{1}{d}.
\end{align}

Finally, by  M\"obius inversion formula, we have
\begin{align}
f_m^{\tau}(1)=\sum_{k}N_{m,k}(\tau)=\sum_{d|m}\frac{\mu(\frac{m}{d})}{\left(\frac{m}{d}\right)^2}
K_{d,0,\frac{d}{2}}^{(\mathbb{C}^3,\mathcal{D}_\tau)},
\end{align}
where $\mu(d)$ is the M\"obius function.

Therefore, by the expression (\ref{Km0})  we obtain
\begin{proposition}
For $m\geq 1$, the value $f_m^{\tau}(1)$ of (\ref{integralityC3in})
at $q=1$ is given by
\begin{align}
f_m^{\tau}(1)=\frac{1}{m^2}\sum_{d|m}\mu(m/d)(-1)^{d\tau}\binom{d(\tau+1)-1}{d-1}.
\end{align}
\end{proposition}

In the following, we will prove that
\begin{theorem} \label{theoremintegrality}
For any $m\geq 1$ and $\tau\in \mathbb{Z}$,
\begin{align} \label{OVnumber}
\frac{1}{m^2}\sum_{d|m}\mu(m/d)(-1)^{d\tau}\binom{d(\tau+1)-1}{d-1}
\in \mathbb{Z}.
\end{align}
\end{theorem}

For $\tau\leq 0$, we have shown the integrality of Ooguri-Vafa
invariant $N_{m,k}(\tau)$ in Theorem \ref{mainthm}, so we only need
to prove Theorem \ref{theoremintegrality} for the case of $\tau>0$
in the following.

 In the author's joint work with W. Luo \cite{LZ}, we develop
a systematic method to deal with the integrality of BPS numbers from
string theory. We can apply our method directly to prove Theorem
\ref{theoremintegrality}.

We define the following function, for nonnegative integer $n$ and
prime number $p$,
\begin{align}
    f_p(n)=\prod_{i=1,p\nmid i}^n i = \frac{n!}{p^{[n/p]}[n/p]!}.
\end{align}

\begin{lemma}[cf. Lemma 4.6 in \cite{LZ}]  \label{functionfp}
    For odd prime numbers $p$ and $\alpha\geq 1$ or for $p=2$, $\alpha\geq 2$,
    we have $p^{2\alpha}\mid f_p(p^{\alpha} n)-f_p(p^{\alpha})^n$. For $p=2, \alpha=1$, $f_2(2n)\equiv (-1)^{[n/2]}\pmod{4}$.
\end{lemma}
\begin{proof}
 With $\alpha\geq 2$ or $p>2$, $p^{\alpha-1}(p-1)$ is even, then
 \begin{align*}
     &f_p(p^\alpha n)-f_p(p^\alpha(n-1))f_p(p^\alpha) \\
     &= f_p(p^\alpha(n-1))\left(\prod_{i=1,p\nmid i}^{p^\alpha} (p^{\alpha}(n-1)+i)-f_p(p^\alpha)\right) \\
     & \equiv p^{\alpha}(n-1) f_p(p^\alpha(n-1))f_p(p^\alpha)\left(\sum_{i=1,p\nmid i}^{p^\alpha}\frac{1}{i}\right) \pmod{p^{2\alpha}} \\
     & \equiv  p^{\alpha}(n-1) f_p(p^\alpha(n-1))f_p(p^\alpha)\left(\sum_{i=1,p\nmid i}^{[p^\alpha/2]} (\frac{1}{i}+\frac{1}{p^{\alpha}-i})\right) \pmod{p^{2\alpha}} \\
     & \equiv  p^{\alpha}(n-1) f_p(p^\alpha(n-1))f_p(p^\alpha)\left(\sum_{i=1,p\nmid i}^{[p^\alpha/2]} \frac{p^{\alpha}}{i(p^\alpha-i)} \right)\equiv 0, \pmod{p^{2\alpha}}
\end{align*}

Thus the first part of the Lemma \ref{functionfp} is proved by
induction. For $p=2, \alpha=1$, the remain argument is
straightforward.
\end{proof}
\begin{lemma}
    For prime number $p$ and $n=p^\alpha a, p\nmid a, \alpha\geq 1, \tau\geq 0$, $p^{2\alpha}$ divides
    \begin{equation}
        (-1)^{\tau n}\binom{(\tau+1)n-1}{n-1}-(-1)^{\tau n/p}\binom{(\tau+1)n/p-1}{n/p-1}. \nonumber
    \end{equation}
    \label{lemma2}
\end{lemma}
\begin{proof}
By a straightforward computation, we have
    \begin{align}
        &(-1)^{\tau n}\binom{(\tau+1)n-1}{n-1}-(-1)^{\tau n/p}\binom{(\tau+1)n/p-1}{n/p-1} \nonumber \\
       &= (-1)^{\tau n}\binom{(\tau+1)n/p-1}{n/p-1}\left( \frac{f_p((\tau+1)n)}{f_p(\tau n)f_p(n)} -(-1)^{\tau(n-n/p)}\right). \label{eq2}
    \end{align}
    For $p>2$ or $p=2, \alpha>1$, then $n-n/p$ is even, thus (\ref{eq2}) is divisible by $p^{2\alpha}$ by Lemma~\ref{functionfp}.
    For $p=2, \alpha=1$, (\ref{eq2}) is divisible by 4 if
 \begin{equation}
        [\frac{(\tau+1)n}{4}]+[\frac{\tau n}{4}]+[\frac{n}{4}] -\tau(n-\frac{n}{2})\equiv 0, \pmod{2} \nonumber
    \end{equation}
    which depends only on $\tau\pmod{2}$, verify for $\tau\in \{0,1\}$ to get the results.
\end{proof}

Now, we can finish the proof of Theorem \ref{theoremintegrality}.
\begin{proof}
For $m\geq 1$, we write $m=p_1^{\alpha_1}p_2^{\alpha_2}\cdots
p_r^{\alpha_r}$, where each $\alpha_i\geq 1$ and $p_1,...,p_r$ are
$r$ distinct primes.

By the definition of M\"obius function, only when
$m/d=p_1^{\delta_1}p_2^{\delta_2}\cdots p_r^{\delta_r}$ for
$\delta_i\in \{0,1\}$, $\mu(m/d)$ is nonzero and
\begin{align}
\mu(p_1^{\delta_1}p_2^{\delta_2}\cdots
p_r^{\delta_r})=(-1)^{\sum_{i=1}^{r}\delta_i}.
\end{align}
Therefore,
\begin{align} \label{divisibleformula}
\sum_{d|m}\mu(m/d)(-1)^{d\tau}\binom{d(\tau+1)-1}{d-1}=\sum_{\delta_i\in
\{0,1\},1\leq i\leq r}(-1)^{\sum_{i=1}^{r}\delta_i}(-1)^{\tau
n_\delta}\binom{(\tau+1)n_\delta-1}{n_\delta-1},
\end{align}
where $n_\delta=p_1^{\alpha_1-\delta_1}p_2^{\alpha_2-\delta_2}\cdots
p_r^{\alpha_r-\delta_r}$. We need to show for any $1\leq i\leq r$,
(\ref{divisibleformula}) is divisible by $p_i^{2\alpha_i}$. Without
loss of generality, we only show that (\ref{divisibleformula}) is
divisible by $p_1^{2\alpha_1}$ in the following.

Indeed, (\ref{divisibleformula}) is equal to
\begin{align}
\sum_{\delta_j\in \{0,1\},j\geq
2}(-1)^{\delta_2+\cdots+\delta_r}\left((-1)^{\tau
n_{\delta'}}\binom{(\tau+1)n_{\delta'}-1}{n_{\delta'}-1}-(-1)^{\tau
n_{\delta'}/p_1}\binom{(\tau+1)n_{\delta'}/p_1-1}{n_{\delta'}/p_1-1}\right),
\end{align}
where $n_{\delta'}=p_1^{\alpha_1}p_2^{\alpha_2-\delta_2}\cdots
p_r^{\alpha_r-\delta_r}=p^{\alpha_1}a$. Therefore, by Lemma
\ref{lemma2}, the formula (\ref{divisibleformula}) is divisible by
$p_1^{2\alpha_1}$, hence we complete the proof of Theorem
\ref{theoremintegrality}.
\end{proof}

\textbf{Acknowledgements.} The author would like to thank Professor
Ole Warnaar for useful discussions \cite{Wa}, and showing him some
insights about the formula (\ref{deformed-r-r}).

$$ \ \ \ \ $$


\begin{thebibliography}{999}

\bibitem{And} G. E. Andrews, {\em Partially ordered sets and the
Rogers-Ramanujan identities}, Aequationes Math. 12 (1975), 94-107.

\bibitem{AKMV} M. Aganagic, A. Klemm, M. Mari\~no and
C. Vafa, {\em The topological vertex}, Comm. Math. Phys. {\bf 254}
(2005), no. 2, 425-478.

\bibitem{AM} P. Aspinwall and D. Morrison, {\em Topological field theory and
rational curves}, Comm. Math. Phys. 151 (1993), 245-262.

\bibitem{AV} A. Aganagic and C. Vafa, {\em  Mirror symmetry, D-branes and
counting holomorphic discs}. arXiv: hep-th/0012041.

\bibitem{AKV} A. Aganagic, A. Klemm and C. Vafa, {\em Disk instantons, mirror
symmetry and the duality web}. Z. Naturforsch. A. 57(1-2), 1-28
(2002).

\bibitem{Baxter} R. J. Baxter, {\em The hard hexagon model and the Rogers-Ramanujan
identities}, in ¡°Exactly Solved Models in Statistical Mechanics,¡±
Chapter 14, Academic Press, London, in press.

\bibitem{BF} K. Behrend and B. Fantechi, {\em The intrinsic normal cone},
Invent. Math. 128 (1997) 45-88.

\bibitem{Bre} D. M. Bressoud, {\em An easy proof of the Rogers-Ramanujan
identities}, J. Number Theory 16 (1983), 235-241.

\bibitem{Bou} V. Bouchard, {\em Lectures on complex geometry, Calabi-Yau
manifolds and toric geometry}, arXiv:hep-th/0702063.


\bibitem{BKMP} V. Bouchard, A. Klemm, M. Mari\~no and S. Pasquetti, {\em Remodeling the B-model}, Commun. Math.
Phys. {\bf 287}, 117-178 (2009).

\bibitem{CB} W. Crawley-Boevey, {\em Geometry of the moment map for representations
of quivers}, Compositio Math. 126 (2001), 257-293.

\bibitem{CBV} W. Crawley-Boevey and M. Van den Bergh, Absolutely indecomposable
representations and Kac-Moody Lie algebras, Invent. Math. 155
(2004), 537-559.

\bibitem{CDGP} P. Candelas, X. C. De La Ossa, P. S. Green and L. Parkes, {\em Pair
of Calabi- Yau manifolds as an exactly soluble superconformal
theory}, Nucl. Phys. B 359 (1991) 21.


\bibitem{CDDP} Wu-yen Chuang, D.-E Diaconescu, R. Donagi and T. Pantev, {\em
Parabolic refined invariants and Macdonald polynomials},
arXiv:1311.3624

\bibitem{CHM} M. A. A. de Cataldo, T. Hausel and L. Migliorini, {\em Topology
of Hitchin systems and Hodge theory of character varieties: the case
$A_1$}. Ann. of Math. (2), 175(3):1329-1407, 2012.

\bibitem{Dia} D.-E. Diaconescu,  {\em Local curves, wild character varieties,
and degenerations}, arXiv:1705.05707.

\bibitem{DDP} D.-E. Diaconescu, R. Donagi and T. Pantev, {\em BPS
states, torus links and wild character varieties}, arXiv:1704.07412.



\bibitem{EO1} B. Eynard and N. Orantin, {\em Invariants of algebraic curves and
topological expansion}, arXiv:math-ph/0702045.

\bibitem{EO2} E. Eynard and N. Orantin, {\em Computation of open Gromov-Witten
invariants for toric Calabi- Yau 3-folds by topological recursion, a
proof of the BKMP conjecture}, Comm. Math. Phys. 337 (2015), no. 2,
483-567.

\bibitem{FLZ} B. Fang, C.-C. M. Liu and Z. Zong, {\em On the remodeling
conjecture for toric Calabi-Yau 3-orbifolds}, arXiv:1604.07123.

\bibitem{GH} A. M. Garsia and M. Haiman, {\em A remarkable q; t-Catalan sequence
and q-Lagrange inversion}, J. Algebraic Combin. 5 (1996), 191-244.

\bibitem{GM} A. Garsia and S. Milne, {\em Method for constructing bijections
for classical partition identities}, Proc. Nat. Acad. Sci. U.S.A. 18
(1981), 2026-2028.


\bibitem{GV1} R. Gopakumar and C. Vafa, {\em M-theory and topological strings-II},
arXiv:hep-th/9812127.

\bibitem{GV2} R. Gopakumar and C. Vafa, {\em On the gauge theory/geometry
correspondence}, Adv. Theor. Math. Phys.3(5) (1999) 1415-1443.


\bibitem{GOW} M. J. Griffin, K. Ono and S. O. Warnaar {\em A framework of Rogers-Ramanujan identities and their arithmetic
properties}, Duke Math. J. Volume 165, Number 8 (2016), 1475-1527.

\bibitem{Hardy} G. H. Hardy, {\em Ramanujan}, Cambridge Univ. Press, London. 1940;
reprinted by Chelsea, New York, 1959.

\bibitem{HKKPTVVZ} K. Hori, S. Katz, A. Klemm, R. Pandharipande, R. Thomas, C. Vafa, R. Vakil and E. Zaslow,
{\em Mirror symmetry}, Clay mathematics monographs. 1.

\bibitem{Hua} J. Hua, {\em Counting representations of quivers over finite fields},
J. Algebra 226 (2000), 1011-1033.

\bibitem{Hau} T. Hausel, {\em Kac's conjecture from Nakajima quiver varieties},
Invent. Math. 181 (2010), 21-37.

\bibitem{HLRV1} T. Hausel, E. Letellier, and F. Rodriguez-Villegas, {\em Arithmetic
harmonic analysis on character and quiver varieties}, Duke Math. J.
160(2011), 323-400.

\bibitem{HLRV2} T. Hausel, E. Letellier and F. Rodriguez-Villegas, {\em Positivity for Kac polynomials and DT-invariants of
quivers}, Ann. of Math. 177 (2013), 1147-1168.

\bibitem{HMW} T. Hausel, M. Mereb, and M. L. Wong, {\em Arithmetic and
representation theory of wild character varieties},
arxiv:1604.03382.


\bibitem{HST} S. Hosono, M. Saito and A. Takahashi, {\em Relative Lefschetz
actions and BPS state counting}, Internat. Math. Res. Notices 15
(2001), 783-816.

\bibitem{IP} E. N. Ionel and T. H. Parker, {\em The Gopakumar-Vafa formula for
symplectic manifolds}, preprint, arXiv:1306.1516.

\bibitem{Kac} V. G. Kac, {\em Root systems, representations of quivers and invariant
theory}, in Invariant Theory (Montecatini, 1982), Lecture Notes in
Math. 996, Springer-Verlag, New York, 1983, pp. 74-108.

\bibitem{Kac2} V. G. Kac, {\em Infinite dimensional Lie algebras}, 3d
ed., Cambridge Univ. Press, Cambridge, 1990.

\bibitem{Konishi} Y. Konishi, {\em Integrality of Gopakumar-Vafa invariants of toric
Calabi-Yau threefolds}, Publ. Res. Inst. Math. Sci. 42 (2006), no.
2, 605-648.

\bibitem{KL} S. Katz, C.-C.M. Liu, {\em Enumerative geometry of stable maps with
Lagrangian boundary conditions and multiple covers of the disc}.
Adv. Theor. Math. Phys. 5(1), 1-49 (2001).

\bibitem{KiL} Y. H. Kiem and J. Li, {\em Categorication of Donaldson-Thomas
invariants via perverse sheaves}, preprint, arXiv:1212.6444.

\bibitem{Kir}  A. Kirillov, Jr., Quiver representations and quiver varieties,
Graduate Studies in Mathematics, 174. American Mathematical Society,
Providence, RI, 2016. xii+295 pp.

\bibitem{Kr} P. B. Kronheimer, {\em The construction of ALE spaces as a hyper-Kahler quotients},
J. Differential Geom. 29 (1989), 665-683.

\bibitem{KN} P. B. Kronheimer and H. Nakajima, {\em Yang-Mills instantons on ALE gravitational instantons},
Math. Ann. 288 (1990), 263-307.

\bibitem{KS} P. Kucharski and P. Sulkowski, {\em BPS counting for knots and
combinatorics on words}, arXiv:1608.06600.


\bibitem{Lus} G. Lusztig, {\em On quiver varieties}, Adv. Math. 136 (1998), 141-182.

\bibitem{LLLZ} J. Li, C.-C. Liu, K. Liu and J. Zhou, {\em A mathematical theory of the topological vertex},  Geometry and Topology 13 (2009)
527-621.

\bibitem{LM1} J.M.F. Labastida and M. Mari\~no, {\em Polynomial invariants for torus
knots and topological strings} Comm. Math. Phys. {\bf 217}
(2001),no. 2, 423.

\bibitem{LM2} J.M.F. Labastida and M. Mari\~no, {\em A new point of view in the theory
of knot and link invariants} J. Knot Theory Ramif. {\bf 11} (2002),
173.

\bibitem{LMV} J.M.F. Labastida, M. Mari\~no and C. Vafa, {\em Knots, links and branes
at large N}, J. High Energy Phys. 2000, no. {\bf 11}, Paper 7.

\bibitem{LP} K. Liu and P. Peng, {\em Proof of the
Labastida-Mari\~no-Ooguri-Vafa conjecture}. J. Differential Geom.,
85(3):479-525, 2010.

\bibitem{LS} J. Li and Y. Song, {\em Open string instantons and relative stable
morphisms}. In: The interaction of finite-type and Gromov-Witten
invariants (BIRS 2003), Volume 8 of Geom. Topol.Monogr., Coventry:
Geom. Topol. Publ., 2006, pp. 49-72.

\bibitem{LT} J. Li and G. Tian, {\em Virtual moduli cycle and Gromov-Witten
invariants of algebraic varieties}, J. Amer. Math. Soc. 11, no. 1,
(1998) 119-174.

\bibitem{LW} J. Lepowsky and R. L. Wilson, {\em The Rogers-Ramanujan
identities: Lie theoretic interpretation and proof}, Proc. Nat.
Acad. Sci. USA. 78 (1981). 699-701.

\bibitem{LLZ} C.-C. Liu, K. Liu and J. Zhou, {\em A proof of a conjecture of Mari\~no-Vafa on Hodge integrals},
J. Differential Geom. {\bf 65}(2003).

\bibitem{LZ} W. Luo and S. Zhu, {\em Integrality structures in topological
strings I: framed unknot}, arXiv:1611.06506.


\bibitem{Mac} I. G. MacDolnald, {\em Symmetric functions and Hall
polynomials}, 2nd edition, Charendon Press, 1995.

\bibitem{Mar}  M. Mari\~no, {\em open string amplitudes and large order
behavior in topological string theory}, arXiv:hep-th/0612127.

\bibitem{Moz} S. Mozgovoy, {\em Motivic Donaldson-Thomas invariants and Kac
conjecture}, 2010. arXiv 1103.2100.


\bibitem{MV} M. Mari\~no, C. Vafa, {\em Framed knots at large N}, in:
Orbifolds Mathematics and Physics, Madison, WI, 2001, in: Contemp.
Math., vol.310, Amer. Math. Soc., Providence, RI, 2002, pp.185-204.

\bibitem{MT} D. Maulik and Y. Toda. {\em Gopakumar-Vafa invariants via
vanishing cycles}. preprint, arXiv:1610.07303.

\bibitem{MMMS} A. Mironov, A. Morozov, An. Morozov, A. Sleptsov, {\em Gaussian
distribution of LMOV numbers}, arXiv:1706.00761.


\bibitem{N1} H. Nakajima, {\em Instantons on ALE spaces, quiver varieties, and
Kac-Moody algebras}, Duke Math. J. 76 (1994), 365-416.

\bibitem{N2} H. Nakajima, {\em Quiver varieties and Kac-Moody algebras}, Duke
Math. J. 91 (1998), 515-560.

\bibitem{OP} A. Okounkov and R. Pandharipande, {\em Hodge integrals and
invariants of the unknot}, Geom. Topol. 8, 675-699 (2004).

\bibitem{OV} H. Ooguri and C. Vafa, {\em Knot invariants and topological strings}.
Nucl. Phys. B 577(3), 419-438 (2000).

\bibitem{Peng} P. Peng, {\em A Simple Proof of Gopakumar-Vafa Conjecture for Local
Toric Calabi-Yau Manifolds}, Commun. Math. Phys. 276 (2007),
551-569.

\bibitem{PSW} R. Pandharipande, J. Solomon and J. Walcher {\em Disk enumeration on the quintic 3-fold},
J. Amer. Math. Soc. Vol 21, Number 4, (2008), 1169-1209.

\bibitem{PT} R. Pandharipande and R. P. Thomas. {\em Curve counting via stable
pairs in the derived category}. Invent. Math., 178(2):407-447, 2009.

\bibitem{Rogers} L. J. Rogers, {\em Second memoir on the expansion of certain
infinite products}, Proc. London Math. Sot. 25 (1894) 318-343.

\bibitem{Schur} J. Schur, {\em Ein Beitrag zur addiven Zahlentheorie}. Sitzungsber.
Preuss. Akad. Wiss. Phw-Math. Kl. (1917), 302-321.

\bibitem{Stem} J.R. Stembridge, {\em Hall-Littlewood functions, plane partitions,
and the Rogers-Ramanujan identities}, Trans. Amer. Math. Soc. 319
(1990) 469-498.


\bibitem{Wa} S.O. Warnaar, Private Communications, 2017.


\bibitem{Witten} E. Witten, {\em Topological Sigma Models}, Commun. Math. Phys.
118, 411 (1988).

\end{thebibliography}
\end{document}